\documentclass[draft, 12pt, reqno]{amsart} 
\usepackage{amssymb,latexsym,amsfonts,verbatim,amscd,ifthen} 
\usepackage{color}
\usepackage{relsize}
\usepackage[utf8]{inputenc}
\usepackage{mathrsfs}
\usepackage{mathbbol}

\newcommand{\C}{{\mathbb C}}
\newcommand{\R}{{\mathbb R}}

\newcommand{\Po}{{\mathcal P}}

\newcommand{\Pp}{{\mathbb P}}

\numberwithin{equation}{section}
\newcommand{\pr}{{\mathcal P}}
\theoremstyle{plain}

\newtheorem{theorem}{Theorem}[section]
\newtheorem{lemma}[theorem]{Lemma}
\newtheorem{proposition}[theorem]{Proposition}
\newtheorem{corollary}[theorem]{Corollary}

\newtheorem*{definition*}{Definition}

\newtheorem{definition}[equation]{Definition}

\begin{document}
\title{ Kippenhahn's conjecture revisited}
\author{Michael  Stessin}
\address{Department of Mathematics and Statistics \\
University at Albany, SUNY \\
Albany, NY 12222}
\email{mstessin@albany.edu}
\begin{abstract}
 In 1951 paper \cite{Ki} Kippenhahn  conjectured that if the characteristic polynomial \  $P_A(x_1,x_2,x_3)=\mbox{det}(x_1A_1+x_2A_2-x_3I)$, \ where $A_1$ and $A_2$ are $n\times n$ Hermitian matrices, has a repeated
factor in the polynomial ring $\C[x_1,x_2,x_3]$, then the pair $(A_1,A_2)$ is unitary equivalent to a direct sum $(C_1\oplus C_2, \ D_1\oplus D_2)$ where $C_i, D_i\in M_{n_i}(\C)  $ for some $1\leq n_i<n, \ n_1+n_2=n,  i=1,2$. Kippenhahn verified the conjecture whenever
the degree of the minimal polynomial of $x_1A_1 + x_2A_2$ is 1 or 2. In subsequent
works \cite{Sh1,Sh2} Shapiro obtained a number of results which supported the
conjecture. In particular, she showed that it held if $n \leq 5$. In 1983 Laffey \cite{La} showed that, in general, Kippenhahn's conjecture was not true by constructing a counterexample for $n=8$.
Since then additional counterexamples were worked out (see \cite{Wa} for example). Some positive results in this direction including the quantum version of the conjecture can be found in \cite{F1, F2, KVo1, Law}.

In this paper we use methods of recently developed local spectral analysis to give some necessary and sufficient conditions for the affirmative answer to Kippenhahn's conjecture in terms of the characteristic polynomials of certain elements of the algebra generated by the matrices in the tuple.

\end{abstract}

\maketitle

\section{Introduction}

Given a set $A_1,...,A_m$ of square $N\times N$ matrices, the characteristic polynomial of the linear combination, $det(x_1A_1+\cdots +x_mA_m-x_{m+1}I)$, (such linear combination is called \textit{a pencil}) is a homogeneous polynomial of degree $N$ in variables $x_1,...,x_{m+1}$.  The zero set  of this polynomial  determines an algebraic variety in the projective space $\C\Pp^{m}$, which is called \textit{the determinantal variety of the pencil}. 

In this paper we  denote the characteristic polynomial and determinantal variety of a pencil determined by $A=(A_1,...,A_m)$ by
\begin{eqnarray*}
& P_A(x_1,...,x_{m+1})=det(x_1A_1+\cdots +x_mA_m-x_{m+1}I), \\
&\sigma(A)=\sigma(A_1,...,A_m) \\
&=\big\{ [x_1:\cdots :x_{m+1}]\in \C\Pp^{m}: \ 	P_A(x_1,...,x_{m+1})=0\big\}
\end{eqnarray*}
respectively. 

%To avoid trivial redundancies it is frequently assumed that at least one of the matrices $A_1,...,A_m$ is invertible, and, thus, can be assumed to be the identity (or rather $-I$). In what follows we always assume that $A_{m+1}=-I$ and write $P_A(x_1,...,x_{m+1})=det(x_1A_1+\cdots +x_mA_m-x_{m+1}I)$ and $\sigma(A)=\big\{ [x_1:\cdots :x_m:x_{m+1}]\in \C\Pp^{m}: \ det(x_1A_1+\cdots +x_mA_m-x_{m+1}I)=0\big\}$.

Linear pencils of matrix tuples and their determinantal varieties  have been under scrutiny for a long time. They appear in various areas of mathematics and physics. Here are just a few of such areas.
\begin{itemize}
\item Notably, the study of group determinants led Frobenius to laying out the foundation of representation theory. 
\item For matrix pairs Kronecker \cite{Kr} established a special form that is called Kronecker Canonical Form. This form is an important tool in  numerical linear algebra, differential-algebraic equations, and control theory.
\item The investigation of the question when an algebraic variety of codimension 1 in a projective space admits a determinantal representation 
%that is there is a pencil which has this variety as its determinantal variety 
was initiated by Dickson \cite{D1}-\cite{D5}, and has been  continuing in the framework of algebraic geometry since then. Without trying to give an exhausting account of the references on this topic, we just mention  \cite{Ca, HV, KV, V}, as well as the monograph \cite{D} and references therein.
\item Joint numerical ranges of Hermitian pencils appear  in problems of
experimental and theoretical physics as they represent quantum states.(see \cite{AS, BZ}). Recent paper \cite{PSW} contains late results and numerous references on the topic.

\item In the last 15 years connections between linear representations of groups and algebras and determinantal varieties of their generators have been studied in the framework of operator theory (see \cite{AY,CCD, CST, GY, GLX, HY, HZ, JL, KVo2, PS1, PS2, S1, S2, Y1}).  In particular, \cite{GY}, \cite{CST}, \cite{PS1} established spectral characterization of representations of infinite dihedral group, of non-special finite Coxeter groups, and of some subgroups of the permutation group related to the Hadamard matrices of Fourier type. In \cite{GLX, JL} a spectral aspect of representations of simple Lie algebras and, in particular,  certain representations of ${\mathfrak s}{\mathfrak l}(2)$ were investigated. 

\end{itemize}

\vspace{.2cm}

An important question appearing in representation theory is when a representation is reducible. For  a finitely generated group $G$ or algebra ${\mathcal A}$ a finite dimensional linear representation $\rho$ is reducible when the matrix tuple of the images of generators under $\rho$ is reducible (decomposable for unitary representations). A investigation of decomposable tuples via the geometry of determinantal varieties of the corresponding pencils was originated in early 1950-s and is still continuing.  (see \cite{Ki, F1, F2, MT1, MT2, KVo1, KVo2, La, Law, Sh1, Sh2, CSZ, MQW, S, SY}). 

If a tuple $A_1,...A_m$ has a common invariant subspace, the characteristic polynomial of the corresponding pencil has a nontrivial factor, and, respectively, the determinantal variety has a proper component. The converse is not true: there are simple examples of matrix tuples without common invariant subspaces whose determinantal varieties have proper algebraic components.  (see \cite{KVo2, ST} for examples). The problem of finding necessary and sufficient conditions for the existence of common reducing subspaces  was considered in \cite{KV,S}. In recent papers \cite{S,SY} some necessary and sufficient conditions for a non-trivial proper component of the determinantal variety of a matrix pencil to correspond to a common invariant subspace were expressed in terms of determinantal varieties of certain elements of the $C^*$-algebra (just the algebra for Hermitian tuples) generated by the underlying tuple. Under these conditions, if a determinantal variety has a non-trivial component of degree $n$ which has multiplicity $k$, then the underlying tuple has a common invariant subspace of dimension $nk$. 

Now, it is natural to ask whether a further reduction is possible. This question for matrix pairs was originally considered by Kippenhahn in his 1951 paper \cite{Ki}. There Kippenhahn conjectured that if the characteristic polynomial   $P_A(x_1,x_2,x_2)=det(x_1A_1+x_2A_2-x_3I)$, where $A_1$ and $A_2$ are $n\times n$ Hermitian matrices, has a repeated
factor in the polynomial ring $\C[x_1,x_2,x_3]$, then the pair $(A_1,A_2)$ is unitary equivalent to a direct sum $(C_1\oplus C_2, \ D_1\oplus D_2)$ where $C_i, D_i\in M_{n_i}(\C)  $ for some $1\leq n_i<n, \ n_1+n_2=n,  i=1,2$. Kippenhahn verified the conjecture whenever the minimal polynomial of $x_1A_1 + x_2A_2$ has degree 1 or 2. In subsequent
works \cite{Sh1,Sh2}, Shapiro obtained a number of results which supported the
conjecture. In particular, she showed that it held, if $n \leq 5$. In 1983 Laffey \cite{La} showed that, in general, Kippenhahn's conjecture was not true by constructing a counterexample for $n=8$.
Since then additional counterexamples were worked out (see \cite{Wa} for example). In 2017 Klep and Vol\u{c}i\u{c} \cite{KVo1} proved the validity of a quantum version of Kippenhahn conjecture. Additional examples can be found in \cite{Law}.

In this paper we return to the original Kippenhahn's conjecture for Hermitian matrix tuples of arbitrary length. Let $A=(A_1,...,A_m)$ be a tuple of Hermitian matrices with $P_A(x_1,...,x_{m+1})=(R(x_1,...,x_{m+1}))^k$ where $R$ is a homogeneous polynomial of degree $n$. We use local spectral analysis to find necessary and sufficient conditions for the tuple $A$ to be unitary equivalent to a tuple that is a direct sum of $k$ copies of a tuple of $n$-dimensional matrices.

The structure of this paper is as follows. In section \ref{background} we assembled necessary background results that are used in the proof of the main theorem. Section \ref{main} is devoted to the proof of this theorem, which is done in several steps. First we do it for $m=2$ under some regularity conditions. We then extend it to the case of an arbitrary $m$ under these conditions. Finally, we establish the result for general tuples as a corollary to the one for regular tuples.

\vspace{.3cm}

\vspace{.3cm}

\section{\textbf{Background results}}\label{background}

\subsection{Local spectral analysis}\label{spec}

Our main tool for proving results of this paper is the local spectral analysis that was developed for various settings in \cite{PS1}, \cite{S1}, \cite{S3}, \cite{ST} in both finite- and infinite-dimensional cases. Strictly speaking the version we need here, namely, the case when one of the operators is self-adjoint, while the others may not be, did not appear in either of these works. The argument in this case follows more or less the same lines. To make our presentation self-contained we present it here. In this paper we are dealing with matrix tuples, so to simplify our presentation we will concentrate only on this setting, even though a similar result  is valid in infinite dimensions as well.

%\vspace{.3cm}

\vspace{.3cm}

Let $A_1,...,A_m$ be $N\times  N$ complex matrices, and assume that $A_1$ is self-adjoint with respect to the standard Hermitian metric on $\C^N$. Remark that adding to a matrix a scalar multiple of the identity does not change the lattice of its invariant subspaces. For this reason without loss of generality we may assume that $A_1$ is invertible, and, hence, the intersection of $\sigma(A_1,...A_m)$ with the projective line $\{[x_1:0:0:\cdots :0:x_{m+1}]\}$ lies in the chart $U_{m+1}=\{ x_{m+1}\neq 0\}$. Following \cite{ST} we call the part of $\sigma(A_1,...,A_m)$ that belongs to this chart \textit{the proper projective joint spectrum} and denote it by $\sigma_p(A_1,...,.A_m)$. Since $U_{m+1}$ is naturally identified with $\C^m$, we define $\sigma_p(A_1,...,A_m)$ to be a subset in $\C^m$:
\begin{eqnarray}
&\sigma_p(A_1,...,A_m) \nonumber \\
&=\{(x_1,...,x_m)\in \C^m: \ \mbox{det}(	x_1A_1+\cdots +x_mA_m- I)=0\} \label{proper spectrum}
\end{eqnarray}

Let $\lambda_1,...,\lambda_n\in \R$ be the set of eigenvalues of $A_1$ and let their multiplicities be $l_1,...,l_n$ respectively. Since $A_1$ is invertible, non of $\lambda_j$ vanishes. Suppose that $\gamma_j=\{w\in \C: \ |w-\lambda_j|=\epsilon_j\}, \ j=1,...,n$ are contours  that separate $\lambda_1,...,\lambda_n$. Then the orthogonal projection ${\mathcal P}_j, \ j=1,...,n$ on the $\lambda_j$-eigenspace of $A_1$ is given by:
\begin{equation}\label{projection for A 1}
{\mathcal P}_j=\frac{1}{2\pi i}\int_{\gamma_j} (w-A_1)^{-1}dw=\frac{1}{2\pi i} \int_{\gamma_j^\prime}(w-\frac{1}{\lambda_j}A_1)^{-1}dw,	
\end{equation}
where $\gamma_j^\prime=\frac{1}{\lambda_j} \gamma_j$. The rank of ${\mathcal P}_j$ is equal to $l_j$, and the spectral resolution of $A_1$ is given by:
\begin{equation}\label{spectral resolution A 1}
	A_1=\displaystyle \sum_{j=1}^n \lambda_j{\mathcal P}_j. 
\end{equation}
The following direct computation appeared in \cite{ST}.

Let $w\in \gamma_j^\prime$.  Then
\begin{eqnarray*}
&(w-\frac{1}{\lambda_j}A_1)^{-1}=	\frac{1}{w-1}{\mathcal P}_j+\displaystyle \sum_{i\neq j} \frac{1}{w-\frac{\lambda_i}{\lambda_j}}{\mathcal P}_i \\
&=	\frac{1}{w-1}{\mathcal P}_j+\displaystyle \sum_{i\neq j} \frac{1}{(w-1)-(\frac{\lambda_i}{\lambda_j}-1)}{\mathcal P}_i \\
&=\frac{1}{w-1}{\mathcal P}_j-\displaystyle \sum_{i\neq j}\Big[ \frac{\lambda_j}{\lambda_i-\lambda_j} \ \frac{1}{1-\frac{\lambda_j(w-1)}{\lambda_i-\lambda_j}}\Big] {\mathcal P}_i\\
&=\frac{1}{w-1}{\mathcal P}_j-\displaystyle \sum_{i\neq j}\displaystyle \sum_{l=0}^\infty \Big(\frac{\lambda_j}{\lambda_i-\lambda_j}\Big)^{l+1}(w-1)^l{\mathcal P}_j\\
&=\frac{1}{w-1}{\mathcal P}_j-\displaystyle \sum_{i\neq j}\displaystyle \sum_{l=0}^\infty \Big(\frac{\lambda_j}{\lambda_i-\lambda_j}\Big)^{l+1}(w-1)^l{\mathcal P}_j^{l+1}\\
&=\frac{1}{w-1}{\mathcal P}_j-\displaystyle \sum_{l=0}^\infty \bigg(\displaystyle \sum_{i\neq j}\frac{\lambda_j}{\lambda_i-\lambda_j}{\mathcal P}_i\bigg)^{l+1}(w-1)^l.
\end{eqnarray*}
Following \cite{ST} we introduce the operator
\begin{equation}\label{T}
{\mathcal T}_j=	\displaystyle \sum_{i\neq j}\frac{\lambda_j}{\lambda_i-\lambda_j}{\mathcal P}_i.
\end{equation}
Then we have
\begin{equation}\label{w-A 1 inverse}
\bigg(w-\frac{1}{\lambda_j}A_1\bigg)^{-1}=\frac{1}{w-1}{\mathcal P}_j-\displaystyle \sum_{l=0}^\infty (w-1)^l{\mathcal T}_j^{l+1}.
\end{equation}

Further, write down 
\begin{eqnarray}
&\sigma_p(A_1,...,A_m) \nonumber \\
&=\big\{x= (x_1:\cdots :x_{m})\in \C^m: \ R_1(x)^{s_1}\cdots R_k(x)^{s_k}=0\Big\}, \label{spectrum}
\end{eqnarray}
where $R_1,...,R_k$ are irreducible  polynomials in $x_1,...,x_{m}$ of degrees $r_1,...,r_k$ respectively. Of course, $l_1+\cdots +l_n=r_1s_1+\cdots +r_ks_k=N$.

We denote by 
$$\Gamma_j=\big\{R_j=0\big\}, \ j=1,...,k$$ 
the zero-set of $R_j$. 
Then $\Gamma_j$ is included in $\sigma_p(A_1,...,A_m)$ as a component of multiplicity $s_j$. 

For each $1\leq j\leq n$ we have $\tau_j=(1/\lambda_j,0,\dots ,0)\in \sigma_p(A_1,...,A_m)$. Suppose that each $\tau_j$ belongs to a single component $\Gamma_i$ and is a regular point of this component (not counting the multiplicity), with  $\frac{\partial R_i}{\partial x_1}\Big|_{\tau_j} \neq 0$. This implies that in a small neighborhood ${\mathcal O}_\epsilon (\tau_j)$ we have $\sigma_p(A_1,...,A_m)\cap {\mathcal O}_\epsilon(\tau_j)=\Gamma_i\cap {\mathcal O}_\epsilon(\tau_j)$.
By the implicit function theorem $\Gamma_i$ is represented in a, perhaps, smaller neighborhood of $\tau_j$ as: 
\begin{equation}\label{implicit}
x_1=x_{1,j}(x_2,...,x_{m}), \ x_{1,j}(0)=1/\lambda_j. 
\end{equation}
If $\hat{x}=(x_2,...,x_m)$ is close to 0, then 
$$(x_{1,j}(\hat{x}),x_2,...,x_m)\in \sigma_p(A_1,...,A_m),$$ 
so that 1 is an eigenvalue of the pencil $A(\hat{x})=x_{1,j}(\hat{x})A_1+x_2A_2+\cdots+x_mA_m$. Similar to \eqref{projection for A 1} 
\begin{equation}\label{projection for pencil}
{\mathcal P}_j(\hat{x})=\frac{1}{2\pi i} \int_{\gamma^\prime}\big(w-A(\hat{x})\big)^{-1}dw	
\end{equation}
is the projection on the subspace spanned by all Jordan cells of $A(\hat{x})$ corresponding to the eigenvalue 1. If $\hat{x}$ is close to 0, the rank of this projection is equal to $l_j$.

At this point we distinguish between the following  cases: 

%\vspace{.2cm}

\textbf{1. $A_2,...,A_m$ are self-adjoint, so that $(A_1,...,A_m)$ is a self-adjoint tuple.}

%\vspace{.2cm}

It follows from Lemma 4.1 in  \cite{S} that in this caee the function $x_{1j}$ given by \eqref{implicit} has real Taylor coefficients. Therefore, if $\hat{x}\in \R^{m-1}$, then $A(\hat{x})=x_{1,j}(\hat{x})A_1+\cdots +x_mA_m$ is self-adjoint, and ${\mathcal P}_j(\hat{x})$ is an orthogonal projection onto $l_j$-dimensional 1-eigenspace of $A(\hat{x})$. We use \eqref{w-A 1 inverse} to obtain
\begin{eqnarray*}
&0=\big(A(\hat{x})-I)\big){\mathcal P}_j(\hat{x})= \frac{1}{2\pi i} \int_{\gamma^\prime} (w-1)\big(w-A(\hat{x})\big)^{-1}dw \\
&=\frac{1}{2\pi i} \int_{\gamma^\prime}(w-1)\Big[ \big(w-\frac{1}{\lambda_j}A_1\big) \\
&-\Big(\displaystyle \sum_{l=2}^m \big(A_l+\frac{\partial x_{1j}}{\partial x_l}(0)A_1\big)x_l+\sum_{s_2,...,s_m}\frac{\partial ^{|s|} x_{1,j}}{\partial x_2^{s_2}\cdots x_m^{s_m} }(0)x_2^{s_2}\cdots x_m^{s_m}A_1\Big]^{-1}dw \\
&=\frac{1}{2\pi i} \int_{\gamma^\prime}(w-1) \big(w-\frac{1}{\lambda_j}A_1\big)^{-1} 
 \Big[I- \Big(\displaystyle \sum_{l=2}^m \big(A_l+\frac{\partial x_{1j}}{\partial x_l}(0)A_1\big)x_l \\
&+\sum_{s_2,...,s_m}\frac{\partial ^{|s|} x_{1,j}}{\partial x_2^{s_2}\cdots x_m^{s_m} }(0)x_2^{s_2}\cdots x_m^{s_m}A_1\Big)\big(w-\frac{1}{\lambda_j}A_1\big)^{-1}\Big]^{-1}dw 
\end{eqnarray*}
\begin{eqnarray*}
&=\frac{1}{2\pi i} \int_{\gamma^\prime}(w-1) \Big( \frac{1}{w-1}{\mathcal P}_j-\displaystyle \sum_{l=0}^\infty (w-1)^l{\mathcal T}_j^{l+1} \Big) \\
&\times  \displaystyle \sum_{r=0}^\infty \bigg(\displaystyle \sum_{l=2}^m \big(A_l+\frac{\partial x_{1j}}{\partial x_l}(0)A_1\big)x_l 
+\sum_{s_2,...,s_m}\frac{\partial ^{|s|} x_{1,j}}{\partial x_2^{s_2}\cdots x_m^{s_m} }(0)x_2^{s_2}\cdots x_m^{s_m}A_1 \Big) \\
& \times \Big( \frac{1}{w-1}{\mathcal P}_j-\displaystyle \sum_{l=0}^\infty (w-1)^l{\mathcal T}_j^{l+1} \Big)\bigg)^r dw.
\end{eqnarray*} 

The last integral is an analytic matrix-valued function in variables $(x_2,...,x_m)\in \C^{m-1}$. Since it vanishes for all $(x_2,...,x_m)\in \R^{m-1}$, it vanishes for all $\hat{x}\in \C^{m-1}$. Now we will rearrange the terms in the integrand to write down the Taylor expansion of this function. To shorten our notation let us write: \\
\begin{eqnarray}
&\alpha_i^j=	A_i-\frac{\partial x_{1j}}{\partial x_l}(0)A_1,   \label{alpha}\\
&\beta_{s_2,...,s_m}^j= \frac{\partial ^{|s|} x_{1,j}}{\partial x_2^{s_2}\cdots x_m^{s_m} }(0)A_1, \ |s|=s_2+\cdots +s_m\geq 2. \label{beta}
\end{eqnarray}
Then we have:
\begin{equation}\label{int}
\frac{1}{2\pi i} \int_{\gamma^\prime}(w-1) \Big [ \frac{1}{(w-1)^2}{\mathcal P}_j+\displaystyle \sum_{(s_2,...,s_m), \ |s|\geq 1}^\infty {\mathcal A}_{s_2,...,s_m}^jx_2^{s_2}\cdots x_m^{s_m}\Big]dw=0,	
\end{equation}
where ${\mathcal A}_{s_2,...,s_m}^j$ is a matrix-valued function of $w$ comprised the following way. We consider all the words of the type
\begin{equation}\label{word}
W= Q_1C_1Q_2C_2\cdots Q_tC_tQ_{t+1}
\end{equation}
with 
\begin{eqnarray*}
&Q_p= \left\{ \begin{array}{cc} {\mathcal P}_j &  \\ & \\ {\mathcal T}_j^u & \mbox{ for some $u\geq  1$}   \
 \end{array}\right.,  \ p=1,...,t+1, \\
 &C_p=\left\{ \begin{array}{cc} \alpha_i^j & \\ & \\ \beta_{q_2,...,q_m}^j & \mbox{for some $(q_2,...,q_m)$}\end{array} \right ., \ p=1,...,t.
 	\end{eqnarray*}
We say that the power of $C_p$ is equal to 1, if $C_p=\alpha_j^j$. Otherwise, the power of $C_p$ is equal to $|q|$.

For each such word $W$ we write
\begin{eqnarray*}	
%&|W|=t \\
&l(W)=\mbox{ the number of ${\mathcal P}_j$ among $Q_1,...,Q_r$} \\
&n(W)=\mbox{the sum of $(u-1)$ for all ${\mathcal T}_j$ terms among $Q_1,...,Q_r$ }\\
&k_i^j(W)=\mbox{the sum of all $\alpha_i^j$ plus the sum of all $q_i$ in $\beta$ terms among $C_p$}, \\
& i=2,...,m.
\end{eqnarray*}
We call the vector $(k_2^j(W),...,k_m^j(W))$ \textit{the signature of $W$} and denote it by $sign(W)$. Obviously, $k_2^j(W)+\cdots +k_m^j(W)$ is equal to the sum of powers of $C_p$.

Given a signature $(a_2,...,a_m)$, the term ${\mathcal A}_{a_2,...,a_m}^j$ is given by
\begin{equation}\label{signature words}
{\mathcal A}_{a_2,...,a_m}^j= \displaystyle \sum_{sign(W)=(a_2,...,a_m),  }\frac{1}{(w-1)^{l(W)-n(W)}}	W
\end{equation}
Let us denote by $\Omega^j(a_2,...,a_m)$ the following collection of words: 
\begin{eqnarray}
&\Omega^j(a_2,...,a_m) \nonumber \\
&=\big\{ W : \ sign(W)=(a_2,...,a_m), \ l(W)-n(W)=1\big\} \label{Omega self}	
\end{eqnarray}

Now, \eqref{int} and the residue theorem imply
$$
\displaystyle \sum_{\Omega(a_2,...,a_m)^j} W=0.	
$$
Write
\begin{eqnarray*}
&\Omega_{{\mathcal P},{\mathcal P}}^j(a_2,...,a_m) = \big \{ W\in \Omega^j(a_2,...,a_m): \ Q_1=Q_{t+1}={\mathcal	 P}_j \big\},\\
&\Omega_{{\mathcal P},{\mathcal T}}^j(a_2,...,a_m)= \big \{ W\in \Omega:^j(a_2,...,a_m): \ Q_1= \ {\mathcal	 P}_j, Q_{t+1}={\mathcal T}_j^u \big\},\\
&\Omega_{{\mathcal T},{\mathcal P}}^j(a_2,...,a_m)= \big \{ W\in \Omega^j(a_2,...,a_m): \ Q_1= \ {\mathcal	 T}_j^u, Q_{t+1}={\mathcal P}_j \big\},\\
&\Omega_{{\mathcal T},{\mathcal T}}^j(a_2,...,a_m)= \big \{ W\in \Omega^j(a_2,...,a_m): \  Q_{1}={\mathcal T}_j^{u_1}, \ Q_{t+1}=\mathcal{T}_j^{u_{t+1}} \big\}
\end{eqnarray*}

Since the range of ${\mathcal P}_j$ is orthogonal to the range of ${\mathcal T}_j$, we have 
%\begin{equation}\label{mpoments1}
\begin{eqnarray*}
&\displaystyle \sum_{\Omega_{{\mathcal P},{\mathcal P}}^j(a_2,...,a_m)} W =\sum_{\Omega_{{\mathcal P},{\mathcal T}}^j(a_2,...,a_m)} W \\
&=\displaystyle \sum_{\Omega_{{\mathcal T},{\mathcal P}}^j(a_2,...,a_m)} W=\sum_{\Omega_{{\mathcal T},{\mathcal T}}^j(a_2,...,a_m)} W =0.	\end{eqnarray*}
%\end{equation}
\vspace{.3cm}

\noindent \textbf{2. At least one of $A_2,...,A_m$ is not self adjoint}.

\vspace{.2cm}

This case is quite similar to the self-adjoint one. The only difference is that, since $A(\hat{x})$ is no longer self-adjoint,  $(A(\hat{x})-I){\mathcal P}_j=0$ may not hold, but rather
$$(A(\hat{x})-I)^{l_j}{\mathcal P}_j=0. $$
Thus,
\begin{eqnarray*}
&\frac{1}{2\pi i} \int_{\gamma^\prime}(w-1)^{l_j} \Big( \frac{1}{w-1}{\mathcal P}_j-\displaystyle \sum_{l=0}^\infty (w-1)^l{\mathcal T}_j^{l+1} \Big) \\
&\times  \displaystyle \sum_{r=0}^\infty \bigg(\displaystyle \sum_{l=2}^m \big(A_l+\frac{\partial x_{1j}}{\partial x_l}(0)A_1\big)x_l 
+\sum_{s_2,...,s_m}\frac{\partial ^{|s|} x_{1,j}}{\partial x_2^{s_2}\cdots x_m^{s_m} }(0)x_2^{s_2}\cdots x_m^{s_m}A_1 \Big) \\
& \times \Big( \frac{1}{w-1}{\mathcal P}_j-\displaystyle \sum_{l=0}^\infty (w-1)^l{\mathcal T}_j^{l+1} \Big)\bigg)^r dw=0.
\end{eqnarray*} 
In this case we replace $\Omega^j(a_2,...,a_m)$ given by \eqref{Omega self}. with 
\begin{equation}\label{Omega not self}
\widehat{\Omega}^j(a_2,...,a_m)=\big\{ W : \ sign(W)=(a_2,...,a_m), \ l(W)-n(W)=l_j\big\},	
\end{equation}
and, respectively define 
$$\widehat{\Omega}_{{\mathcal P},{\mathcal P}}^j(a_2,...,a_m),  \ \widehat{\Omega}_{{\mathcal P},{\mathcal T}}^j(a_2,...,a_m),  \ \widehat{\Omega}_{{\mathcal T},{\mathcal P}}^j(a_2,...,a_m), \ 
\widehat{\Omega}_{{\mathcal T},{\mathcal T}}^j(a_2,...,a_m)$$
the same way as above with the only difference that in all definitions  $W\in \widehat{\Omega}^j(a_2,...,a_m)$.

\vspace{.2cm}

We summarize the above results in the following Theorem which is similar to the ones proved in \cite{PS1,S3,ST}.

\begin{theorem}\label{local spectral}
Let $(A_1,...,A_m)$ be a tuple of complex $N\times N$ matrices, and $A_1$ be self-adjoint and invertible. Suppose that the proper projective joint spectrum of $(A_1,...,A_m)$ is given by \eqref{spectrum}, and that the reciprocal of each eigenvalue of $A_1$  belongs to a single component $\Gamma_i$ and is a regular point of this component (not counting the multiplicity), with  $\frac{\partial R_i}{\partial x_1}\Big|_{\tau_j} \neq 0$.
Then
 \begin{itemize} 
\item[(a)] If $A_2,...,A_m$ are self-adjoint, we have
\begin{eqnarray}
&\displaystyle \sum_{\Omega_{{\mathcal P},{\mathcal P}}^j(a_2,...,a_m)} W =\sum_{\Omega_{{\mathcal P},{\mathcal T}}^j(a_2,...,a_m)} W \nonumber \\
&=\displaystyle \sum_{\Omega_{{\mathcal T},{\mathcal P}}^j(a_2,...,a_m)} W=\sum_{\Omega_{{\mathcal T},{\mathcal T}}^j(a_2,...,a_m)} W =0.	\label{local self}
\end{eqnarray}
\item[(b)]. If at least one of $A_2,...,A_m$ is not self-adjoint
\begin{eqnarray}
&\displaystyle \sum_{\widehat{\Omega}_{{\mathcal P},{\mathcal P}}^j(a_2,...,a_m)} W =\sum_{\widehat{\Omega}_{{\mathcal P},{\mathcal T}}^j(a_2,...,a_m)} W \nonumber \\
&=\displaystyle \sum_{\widehat{\Omega}_{{\mathcal T},{\mathcal P}}^j(a_2,...,a_m)} W=\sum_{\widehat{\Omega}_{{\mathcal T},{\mathcal T}}^j(a_2,...,a_m)} W =0.	\label{local not self}
\end{eqnarray}
\end{itemize}	

\end{theorem}

\vspace{.2cm}

\subsection{Admissible transformations and tuples}

\vspace{.2cm}
In the previous subsection  we noted that without loss of generality we might assume that $A_1$ was invertible. The same argument shows that we might assume that all $A_1,...,A_m$ are invertible. 
In this case the intersection of $\sigma(A_1,...,A_m)$ with the $j$-th coordinate projective line 
lies in $U_{m+1}$ for all $j$, that is  the intersection of the determinantal variety with each projective coordinate line 
$$L_j=\big\{ [0:\cdots :x_j: 0: \cdots :1]: \ x_j\in \C \big\}$$
belongs to $\sigma_p(A_1,...,A_m)$.
%detrminantla variety of the tuple belongs to the chart $\{x_{m+1}\neq 0\}$ in $\C\Pp^m$, which we identify with $\C^m$.

Again, suppose that the proper projective joint spectrum of the tuple  is given by \eqref{spectrum}.
\begin{comment}
\begin{eqnarray}
\sigma_p(A_1,...,A_m)=\big\{x\in \C^m: \ R_1(x)^{m_1} R_2(x)^{m_2}\cdots R_k(x)^{m_k}=0\big\}, \label{polynomial}
\end{eqnarray}
where each $R_s$ is an irreducible polynomial of degree $r_s$. Thus, we have $r_1m_1+\cdots +r_km_k=N$, and 
%The invertibility of $A_1,...,A_n$ implies that for each $1\leq s\leq k$ 
%and $1\leq j\leq n$ 
the intersection of the algebraic variety $\{R_s=0\}\subset \C^{n}$ with the coordinate projective line 
 consists of $r_s$ points counting multiplicity. 
\end{comment}
\vspace{.1cm}

In general, there is no analog of the spectral mapping theorem for determinantal varieties. However, we can give an explicit expression for  the determinantal variety of a tuple that is   the image  of $(A_1,...,A_m)$ under a linear transformation.

For an $m\times m$ complex matrix $C=[c_{ij}]_{i,j=1}^m$, consider the tuple transformation:
\begin{equation}\label{transform}
\widehat{A}_j=\displaystyle \sum_{s=1}^m c_{js}A_s, \ j=1,...,m. 
\end{equation} 

The following observation is straightforward (see \cite{ST}):

\begin{equation}\label{spectrum transform}
\sigma_p(\widehat{A})=\Big\{ \displaystyle \prod_{l=1}^k R_l^{s_l}(C^Tx)=0\Big\},
\end{equation}
%where $\mathscr{C}$ is the $(n+1)\times (n+1)$ matrix obtained from $C$ by adding the $(n+1)$-th  column $(0,...,0,1)^T$ and the $(n+1)$-th row $(0,...,0,1)$.

We remark that:

\begin{itemize}

\item[(1)] If $C$ is invertible, then the polynomials $R_l(C^Tx)$ are irreducible if and only if
$R_l(x)$ are irreducible.

\item[(2)] If $C$ is invertible, then a subspace is invariant under the action of $A_1,...,A_m$ if and only if it is invariant for $\widehat{A}_1,...,\widehat{A}_m$.

\end{itemize}

\begin{definition}\label{admissible}
We call a transformation \eqref{transform} \textit{admissible}, if
\begin{itemize}
\item[(1)] the matrix $C$ is invertible and, if the tuple $A$ consists of self-adjoint matrices, $C$ is real-valued;
\item[(2)] for each $1\leq j\leq n$, at every point of the intersection of $\sigma_p(\widehat{A}) $ with $L_j$, the derivative of \ $ \prod_{l=1}^k R_l(C^Tx)$ with respect to $x_j$ is not equal to 0.
\end{itemize}	
\end{definition}
\noindent Part (2) of the above definition indicates that every point of the intersection of $\sigma_p(\widehat{A})$ with $L_j$ is a regular point of the algebraic variety $\{ \prod_{l=1}^k R_l(C^Tx)=0\}$. In the sequel, tuples with this spectral property will be called {\em admissible tuples}. Since every regular point has multiplicity 1, we see that for admissible tuples the intersection of $\Gamma_s=\big\{R_s(x)=0\}$ with $L_j$ consists of $r_s$ distinct points, and these sets of points are different for different $s$. The following Theorem was proved in \cite[Theorem 1.11]{S} for  self-adjoint tuples, and in \cite[Theorem 2.1]{SY} in general case.

\begin{theorem}\label{admissible tr}
For every tuple of invertible $N\times N$ matrices $A_1,...,A_m$,  admissible transformations form an open and dense set in every neighborhood of the identity matrix $I_m$.	
\end{theorem}

It follows from  Theorem \ref{admissible tr} and our observation above that for every matrix tuple $A$ there is an arbitrarily close to $A$ admissible tuple with the same lattice of common invariant subspaces.

\subsection{Projections}

Let $A$ be a self-adjoint matrix with eigenvalues \newline 
$\lambda_1,...,\lambda_n$.  Equation \eqref{projection for A 1} determines the orthogonal projection $\mathcal{P}_j$ onto the $\lambda_j$-eigenspace of $A$. The following result was proved in \cite[Proposition 3.2 ]{SY}.

\begin{proposition}\label{projections in algebra}
\begin{equation}\label{pr}
{\mathcal P}_j=\frac{1}{\displaystyle \prod_{r\neq j}(\lambda_j-\lambda_r)} \displaystyle \prod_{r\neq j}(A-\lambda_rI). 
\end{equation}
\end{proposition}

\section{\textbf{Main result}}\label{main}

\vspace{.3cm}

Let $A_1,...,A_m$ be self-adjoint $nk\times nk$ matrices., and suppose that 
\begin{equation}\label{small spectrum}
\sigma_p(A)=\Big\{[x_1:\cdots :x_{m}]\in \C^m: \ \big(R(x_1,...,x_{m+1})\big)^k=0 \Big\}, 
\end{equation}
where $R$ is a polynomial of degree $n$.
We will be dealing with the following subset of the free algebra generated by $A_1,...,A_m$. Consider words in this algebra that are in the form
\begin{equation}\label{W}
W=Q_1S_1Q_2S_2\cdots Q_rS_rQ_{r+1}, 
\end{equation}
where $r\leq n-1$,  each $Q_j $ is one of $\{ A_2,...,A_m\}$,  each of $S_j$ is one of $\{ \Po_1,...,\Po_n\}$ given by \eqref{projection for A 1}, and all $S_1,...,S_r$ are different. Proposition \ref{projections in algebra} shows that each such word is a matrix which belongs to the algebra generated by $A_1,...,A_m$.

We denote by $\mathscr{W}$ the collection of all such words. Of course, 
$$A_2,...,A_m  \in \mathscr{W}.$$ 
%We will denote by $\sigma(\mathscr{W})$ and $\sigma_p(\mathscr{W})$ the determinantal variety and the proper projective joint spectrum of the matrices in $\mathscr{W}$.

\vspace{.2cm}

The following Theorem combined with  Corollary \ref{cor} below is our main result. As mentioned in the previous section, without loss of generality we assume that all matrices $A_1,...,A_m$  are invertible.

\begin{theorem}\label{main theorem}
Let $(A_1,...,A_m)$ be an admissible tuple of $nk\times nk$ self-adjoint complex invertible matrices.  
Suppose that the projective joint spectrum, $\sigma_p(A)$ is given by \eqref{small spectrum}, 
where $R$ is a  polynomial of degree $n$ such that $R(x_1,0,...,0)$ has simple roots as a polynomial in $x_1$. 
The tuple $(A_1,...,A_m)$ is unitary equivalent to a tuple that is a direct sum of $k$ identical $m$-tuples of $n\times n$ matrices, if and only if for each $W\in \mathscr{W}$
\begin{eqnarray}
&\sigma_p(A_1,W) 
&=\bigg\{\bigg({\mathscr R}_W(x_1,x_2)\bigg)^k=0\bigg\}, \label{big spectrum}
\end{eqnarray}	
where ${\mathscr R}_W$ is a polynomial of degree $n$.
\end{theorem}

\vspace{1cm}
 
%\vspace{.2cm}
We prove this Theorem in three steps.
\subsection{Necessity}

Suppose that in some orthonormal basis each matrix of the tuple $(A_1,...,A_m)$ is block-diagonal and consists of $k$ identical  diagonal blocks of dimension $n\times n$. Then, obviously, the same is true for the matrices  in $\mathscr{W}$, and, therefore, for each $W\in \mathscr{W}$ \eqref{big spectrum} holds for some polynomial ${\mathscr R}_W$ of degree $n$.

\subsection{Sufficiency }
Since the eigenvalues of $A_1$ are the reciprocals of the roots of $R(x_1,0,...,0)$ the eigenvalues of $A_1$, which we denote by \newpage  
$\lambda_1,...,\lambda_n$, satisfy $\lambda_i\neq \lambda_j, \ i\neq j$ and each of them has multiplicity $k$, so that in an appropriate orthonormal basis $ e_1,...e_{nk}$, \ $A_1$ is a diagonal matrix in the form
\begin{equation}\label{A 1}
A_1=\left [ \begin{array}{cccc}\lambda_1I_k & 0 & \cdots &0\\ 0 & \lambda_2I_k& \cdots & 0\\ \cdot & \cdot & \cdot & \cdot \\ 0 & 0 &\cdots & \lambda_nI_k\end{array}\right ], 
\end{equation}
where $I_k$ is the $k\times k$ identity matrix. In this basis $\mathcal{P}_j$ corresponds to $\lambda_jI_k$ entry in \eqref{A 1}.

We have $\tau_j=(1/\lambda_j,0,...0)\in \sigma_p(A_1,...,A_m)$ and, since the tuple $(A_1,...,A_m)$ is admissible, 
%$$\frac{\partial R}{\partial x_1}\Big|_{\tau_j}\neq 0. $$ 
and since for each $ W\in \mathscr{W}$
$$\eta_j=(1/\lambda_j,0)\in \sigma_p(A_1,W), $$
we see that
\begin{equation}\label{regular W}
\frac{\partial {\mathscr R}_W}{\partial x_1}\Big|_{\eta_j}\neq 0. 
\end{equation}
Indeed, 
$$\frac{\partial {\mathscr R}_W}{\partial x_1}\Big|_{\eta_j}=\frac{\partial R}{\partial x_1}\Big|_{\tau_j}=\lambda_j \displaystyle \prod_{l\neq j} \bigg(\frac{\lambda_l}{\lambda_j}-1\bigg)\neq 0$$

\begin{comment}
Further, the determinantal variety of any subtuple of the tuple  $\mathscr{W}$ is given by the intersection of $\sigma_p(\mathscr{W})$ 
with the coordinate plane where the variables corresponding to matrices not included in the subtuple are set to 0.
Thus, the characterisic polynomial  of every subtuple is also a $k$-th power of some polynomial of degree $n$. 
\end{comment}
\subsubsection{$\bf{m=2}$}

First, we apply Theorem \ref{local spectral} to the proper projective joint spectrum of $(A_1,A_2)$.  Fix $1\leq i\leq n$. The set $\Omega^i_{{\mathcal P}{\mathcal P}}(1)$ consists of the single word ${\mathcal P}_i (A_2+\frac{\partial x_{1j}}{\partial x_2}(0)A_1){\mathcal P}_i$,  so that equation  \eqref{local self} implies
\begin{equation}\label{moment 1 for 2}
\pr_i\big(A_2+ \frac{\partial x_{1i}}{\partial x_2}(0)A_1\big)\pr_i=0.
\end{equation}
Since ${\mathcal P}_iA_1=A_1{\mathcal P}_i={\mathcal P}_i$, we obtain
\begin{equation}\label{first moment for 2}
{\mathcal P}_iA_2{\mathcal P}_i=c_{ii}\Po_i,	
\end{equation}
where $c_{ii}=-\frac{\partial x_{1i}}{\partial x_2}(0)$. 
Thus, the compression of $A_2$ to the range of ${\mathcal P}_i$ is $c_{ii}I_k $.

Let $i\neq j$. Consider the pair $(A_1,W)$ where $W= A_2\Po_j A_2$. Equation \eqref{big spectrum} implies that $\sigma_p(A_1,A_2\Po_jA_2)$ is in the form
$$\sigma_p(A_1,A_2\Po_jA_2)= \{( R_j(x_1,x_2))^k=0\}, $$
where $R_j(x_1,x_2)$ is a polynomial of degree $n$. Equation \eqref{regular W} shows that
$$ \frac{\partial R_j}{\partial x_1}\Big|_{\tau_i}\neq 0,$$
so that Theorem \ref{local spectral} is applicable.
Since $\Po_j$ is self-adjoint, the pair $(A_1,A_2\Po_jA_2)$ is self-adjoint. We apply Theorem \ref{local spectral} part (a) and obtain
$$\Po_iA_2\Po_jA_2\Po_i=-\frac{\partial x_{1i}^j}{\partial x_2}(0)\Po_i, $$
where $ x_1=x_{1i}^j(x_2)$ is the implicit function that represents $\{R_j(x_1,x_2)=0\}$ in a small neighborhood of $(1/\lambda_i,0)$. To shorten our future notation write 
\begin{equation}\label{c i j}
c_{ij}^2 =-\frac{\partial  x_{1,i}^j}{\partial x_2}(0).
\end{equation}
Then, since $\Po_j^2=\Po_j$, 
$$(\Po_i A_2\Po_j)(\Po_i A_2\Po_j)^*= c_{ij}^2\Po_i$$

Now, $(\mathcal{P}_iA_2\mathcal{P}_j)(\mathcal{P}_iA_2\mathcal{P}_j)^*$ is a positive  matrix. Equation \eqref{c i j} shows that $c_{ij}^2\geq 0$ (and that explains our notation). Hence, we might choose  $c_{ij}\geq 0$, and, 
\begin{equation}\label{unitary block} 
\mbox{either} \ \mathcal{P}_iA_2\mathcal{P}_j=0 \ (\mbox{when} \ c_{ij}=0), \ \mbox{or} \ \mathcal{P}_iA_2\mathcal{P}_j=c_{ij}\mathcal{U}_{ij}, 
\end{equation}
where $\mathcal{U}_{ij}$ is a matrix comprised of $n\times n$ blocks of size $k\times k$ each, with all the blocks except for the $(ij)$-th one being 0, and the $(ij)$-th block, which we denote by $u_{ij}$, being a unitary $k\times k$ matrix. 
Since $A_2$ is self-adjoint,  $u_{ji}=u_{ij}^*$. We also set $u_{ii}=I_k$, which is justified by \eqref{first moment for 2}.

If $c_{ij}=0$ for all $i\neq j$, then by \eqref{first moment for 2} the matrix $A_2$ is block diagonal with $k\times k$ diagonal blocks, each being a scalar multiple of $I_k$. In this case the claim of Theorem \ref{main theorem} is obviously true, as here $A_1$ and $A_2$ commute, $\sigma_p(A_1,A_2)$ is comprised of lines, and this is one of the cases worked out by Kippenhahn \cite{Ki} (see also \cite{MT1, MT2}). 

Suppose that $c_{ij}\neq0$ for some pairs $(ij), \ i\neq j$.
It is important to mention that while $c_{ij}$ are defined for all $(ij)	$, at this point 
the matrices $u_{ij}$ are defined only for those pairs $(ij)$ where $c_{ij}\neq 0$, and where $i=j$. We will need them defined for all $(ij)$, which will allow us to construct a unitary transformation that brings $A_2$ to a $k\times k$ block form with all $k\times k$ blocks being diagonal. There is  a flexibility here which allows us to define these matrices up to a unimodular scalar multiple.    We  need 
the following Lemma. 

%\vspace{.2cm}

%\LARGE{Check the statement of Lemma}\normalsize

%\vspace{.5cm}

\begin{lemma}\label{relations}
Let $j_0,...j_r$ be distinct numbers between $1$ and $n$.  	Set $j_{r+1}=j_0 $. If $c_{j_lj_{l+1}}\neq 0, \ t=0,...,r$, then
\begin{equation}\label{cycles}
\displaystyle \prod_{l=0}^r u_{j_l j_{l+1}}=e^{i\theta}I_k .	
\end{equation}
\end{lemma}

\vspace{.1cm}

%Let $i,j,r$ be 3 different numbers between $1$ and $n$. 
\begin{proof}

Let  $j_0,...j_r$ be distinct number between $1$ and $n$. Write 
$$W_{j_1...j_r}=A_2\pr_{j_1}A_2\pr_{j_2}\cdots A_2\pr_{j_r}A_2$$ 
and consider 
$\sigma(A_1, W_{j_1...j_r} )$. Once again, \eqref{regular W} shows that Theorem \ref{local spectral} is applicable, and we apply  its part (b) to  $\widehat{\Omega}_{\pr \pr}^{j_0}(k).$  The only word in this collection is
\small\begin{eqnarray*}
&\pr_{j_0}\big(W_{j_1...,j_r}+\frac{\partial x_{1r}}{\partial x_2}(0)A_1\big)\pr_{j_0} %\big(W_{j_1...,j_r}+\frac{\partial x_{1r}}{\partial x_2}(0)A_1\big)\pr_{j_0}
 \cdots \pr_{j_0}\big(W_{j_1...,j_r}+\frac{\partial x_{1r}}{\partial x_2}(0)A_1\big)\pr_{j_0} \\
&=\Big(\pr_{j_0}W_{j_1...j_r}\pr_{j_0}+\frac{\partial x_{1r}}{\partial x_2}(0)\pr_{j_0}\Big)^k.
\end{eqnarray*}\normalsize

Equation \eqref{local not self}  implies
\begin{equation}\label{nilpotent}
\Big(\pr_{j_0}W_{j_1...j_r}\pr_{j_0}+\frac{\partial x_{1r}}{\partial x_2}(0)\pr_{j_0}\Big)^k=0,	
\end{equation}

Since
\begin{eqnarray*}
&\pr_{j_0}W_{j_1...j_r}\pr_{j_0}=\pr_{j_0}A_2\pr_{j_1}A_2\pr_{j_2}\cdots A_2\pr_{j_r}A_2\pr_{j_0}\\
&=\displaystyle \prod_{l=0}^r\Big( \pr_{j_l}A_2\pr_{j_{l+1}}\Big)=\displaystyle \prod_{l=0}^r \Big(c_{j_l j_{l+1}}\mathcal{U}_{j_l j_{l+1}}\Big) ,\end{eqnarray*} 
 
Equation \eqref{nilpotent} shows that the $k\times k$ dimensional matrix
$$\bigg(\displaystyle \prod_{l=0}^r c_{j_l j_{l+1}}u_{j_l j_{l+1}}\bigg) + \frac{\partial x_{1r}}{\partial x_2}(0)I_k$$
is nilpotent. If  \ $\displaystyle \prod_{l=0}^r c_{j_lj_{l+1}} \neq 0$, it follows that 
\begin{equation}\label{unitary to constant}
\displaystyle \prod_{l=0}^r u_{j_l j_{l+1}} = c I_k+ \mbox{ a nilpotent matrix}, \ c\in \C.	
\end{equation}
The spectral mapping theorem implies that the spectrum  $\sigma\Big(\displaystyle \prod_{l=0}^r u_{j_l j_{l+1}}\Big)$ consists of a single point $c$ having multiplicity $k$. Since $\Big(\displaystyle \prod_{l=0}^r u_{j_l j_{l+1}}\Big)$ is a unitary matrix, $|c|=1$. Now, the spectral theorem yields \eqref{cycles}.
\end{proof}

\vspace{.1cm}

Let us denote by $\mathscr{I}$ the set of all pairs $(i,j)$ where $u_{ij}$ has already been defined (at this point they are the ones with $c_{ij}\neq 0$ plus all pairs $(i,i)$). Since $A_2$ is self-adjoint, $(i,j)\in \mathscr{I}$ if and only if $(j,i)\in \mathscr{I}$.

We will now use the following steps to add additional pairs to $\mathscr{I}$ and define matrices $u_{ij}$ for added pairs. 

Let $(i,s),(i,t)\in \mathscr{I}$. If $(s,t)$ is already in $\mathscr{I}$, we do not add anything to $\mathscr{I}$. If $(s,t)$ is not in $\mathscr{I}$, we define
\begin{equation}\label{amendment}
u_{st}=u_{si}u_{it}, \ u_{ts}=u_{st}^*=u_{ti}u_{is} 
\end{equation}
and add pairs $(s,t)$ and $(t,s)$ to $\mathscr{I}$. Observe, that Lemma \ref{relations} implies that this definition is consistent. Indeed, if $(i,s),(i,t),(j,s),(j,t)\in \mathscr{I}$, then according to \eqref{cycles} 
$$u_{si}u_{it}=e^{i\theta }u_{sj}u_{jt}, $$
so that $c_{st}u_{si}u_{it}=c_{st}u_{sj}u_{jt}$ up to a unimodular scalar multiple. We also observe that \eqref{amendment} and Lemma \ref{relations} imply that relation \eqref{cycles} holds for all $j_0,...,j_r$ such that for all $l$ the pair $(j_l,j_{l+1})$ is in the amended set $\mathscr{I}$.

We continue this process of extending  the set $\mathscr{I}$ and definition of matrices $u_{st}$. In a finite number of steps the set $\mathscr{I}$ stabilizes, and no new pairs can be added to it.  Let us denote by $\widehat{\mathscr{I}}$ this maximal set.

The following Lemma is our next step in the proof of Theorem \ref{main theorem} for $m=2$. 
%\vspace{.2cm}

\begin{lemma}\label{all i j}
There is a partition $\{1,...,n\}= \cup_{l=1}^s	 \mathcal{I}_l$, \ $\mathcal{I}_l=\{i_1^l,...,i_{r_l}^l\}$, \ $l=1,...,s$, \ $r_1+\cdots +r_s=n$ (and, so $\mathcal{I}_l\cap \mathcal{I}_t=\emptyset, \ l\neq t$) such that
$$ \widehat{\mathscr{I}}=\cup_{l=1}^s \widetilde{\mathscr{I}}_l,$$
where $\widetilde{\mathscr{I}}_l$ are given by  $\widetilde{\mathscr{I}}_l=\{ (i_p^l,i_q^l): \ 1\leq p,q\leq r_l\}$.
\end{lemma}

\begin{proof}
Let $i_1^1=\min \{i: \ \exists \ j\neq i \ \mbox{such that} \ (i,j)\in \widehat{\mathscr{I}}\}$, and  $i_2^1< \cdots <i_{r_1}^1$ be all the indices $s$ such that $s\neq i_1^1$ and $(i_1^1,s)\in \widehat{\mathscr{I}}$ (of course, $i_1^1<i_2^1$). Then $(i_l^1,i_1^1)\in \widehat{\mathscr{I}}, \ l=2,...,r_1$. Since $\widehat{\mathscr{I}}$ is maximal, all pairs $(i_s^1,i_t^1)$ are in  $\widehat{\mathscr{I}}, \ s.t=2,...,r$, because by the extension process 
$$u_{i_s^1 i_t^1}=u_{i_s^1 i_1^1}u_{i_1^1 i_t^1} \ \mbox{mod} \ e^{i\theta}. $$
 Also, if for some $2\leq l\leq r_1$ and $s\neq i_1,...,i_r$ the pair $(i_l^1,s)$ is in $\widehat{\mathscr{I}}$, then the maximality of $\widehat{\mathscr{I}}$ implies that $(i_1^1,s)$ must be in it, a contradiction. 
Now, we set $\widetilde{\mathscr{I}}_1=\{(i_s^1,i_t^1): \ s,t=1,...,r\}$. 

%\textbf{Claim} $\widetilde{\mathscr{I}}=\widehat{\mathscr{I}}=\Big\{ (ij): \ 1\leq i,j\leq n\Big\}$.

%We will show that $\widehat{\mathscr{I}}\neq \widetilde{\mathscr{I}}_1$ contradicts the fact that $R$ in \eqref{small spectrum} is irreducible.

Suppose that $\widehat{\mathscr{I}}\neq \widetilde{\mathscr{I}}_1$. Let
$i_1^2=\min \{ i: \ \exists \ j\neq i \ \mbox{such that} \ (i,j)\in \widehat{\mathscr{I}}\setminus \widetilde{\mathscr{I}}_1 \}$, and let $i_2^2,...,i_{r_2}^2$ be all the indices $s\neq i_1^2$ such that $(i_1^2,s)\in \widehat{\mathscr{I}}\setminus \widetilde{\mathscr{I}}_1 \} $. It clearly follows from \eqref{cycles} and the above argument that $i_s^2\neq i_l^1$. Now,  similarly  
%perform all the extension steps which we used for comprising $\widetilde{\mathscr{I}}_1$,
% and obtain
  $\widetilde{\mathscr{I}}_2=\{(i_s^2,i_t^2): \ s,t =1,...,r_2\}$. Evidently $\widetilde{\mathscr{I}}_1\cap \widetilde{\mathscr{I}}_2 =\emptyset$.
  
  Continuing this way we will exhaust the whole $\widehat{\mathscr{I}}$ in a final number of $s$ steps, thus representing $\widehat{\mathscr{I}}$ as a disjoint union $\widehat{\mathscr{I}}=\cup_{p=1}^s \widetilde{\mathscr{I}}_p $, where   each $\widetilde{\mathscr{I}}_p$  has the structure of a lattice, that is there are $i_1^p,...,i_{s_p}^p$ such that $\widetilde{\mathscr{I}}_p=\{(i_\alpha^p,i_\beta ^p)$, \ \ $\alpha, \beta \in \{i_1^p,...,i_{s_p}^p\} $ . The proof of Lemma \ref{all i j} is done.  
\end{proof}

  Finally, for those $(i,j)$ that are not in $\widehat{\mathscr{I}}$ we set $u_{ij}=I_k$. With this setting we see that $k\times k$ block structure of $A_2$ is comprised of $c_{ij}u_{ij}$ blocks, $i,j=1,...,n$. 
Of course, relation \eqref{cycles} holds for  every $j_0,...,j_r$ with $(j_l,j_{l+1})\in \widehat{\mathscr{I}}$. 
We now proceed with the proof of Theorem \ref{main theorem} in the case $m=2$.

Rearranging, if necessary, the numeration of the blocks (that is rearranging the numeration of the basic vectors), we may without loss of generality  think that 
$$(i_1^1,...,i_{r_1}^1)=(1,...r_1), \ 
(i_1^2,...,i_{r_2}^2)=(r_1+1,...,r_1+r_2),..., $$ 
and, of course, after such possible rearrangement $A_1$ is still a block diagonal matrix with each $k\times k$ diagonal block being a scalar multiple of $I_k$,  and for $A_2$ the $k\times k$ block structure is in the form 
\begin{equation}\label{blocks of A 2}
A_2 =\left [\begin{array}{ccccccc} D_1 &0 &\cdot &\cdot  & 0&\cdots& 0 \\ 0 &  D_2 &0  & \cdot  &0&\cdots &0 \\ \cdot & \cdot &\cdot & \cdot & \cdot &\cdot &\cdot \\ 0 & 0  & 0 & D_s & 0 &\cdots &0\\ 0& \cdot &\cdot &\cdot&d_{s+1}I_k&\cdots &0 \\ \cdot &\cdot &\cdot & \cdot &\cdot &\cdot &\cdot \\ 0 & \cdot & \cdot & \cdot & 0&\cdots & d_{n}I_k\end{array} \right ],
\end{equation}
where $d_t=c_{tt}, \ t=r_1+\cdots +r_s,...,n$, and  $D_j$ is the following $kr_j\times kr_j$ matrix consisting of $k\times k$ blocks. Let 
\begin{equation}\label{t j}
t_j=r_1+\cdots+r_{j-1}, \ t_0=0
\end{equation}
then
\begin{equation}\label{D}
D_j=\left [ \begin{array}{ccc}c_{t_j+1,t_j+1}u_{t_j+1, t_j+1}& \cdots & c_{t_j+1, t_{j+1}}u_{t_j+1,t_{j+1}}\\ \cdots & \cdots & \cdots \\ c_{t_{j+1},t_j+1}u_{t_j+1,t_j+1} & \cdots &c_{t_{j+1},t_{j+1}}u_{t_{j+1},t_{j+1}}\end{array}\right ]. \end{equation}

 Write
$$U_j=\left [ \begin{array}{ccccc} u_{t_{j+1},t_j+1}&0& \cdots &0&0 \\0 & u_{t_{j+1},t_j+2} & \cdots &0&0 \\ \cdots & \dots & \cdots & \cdots&0 \\ 0 & 0 &\cdots & u_{t_{j+1},t_{j+1}-1}&0\\ 0 & 0 & \cdots & 0 & I_k \end{array}\right ]. $$

Now,  \eqref{cycles} implies
$$U_jD_jU_j^*= \left [ \begin{array}{ccc}c_{t_j+1,t_j+1}e^{i \theta_{t_j+1,t_j+1}}I_k& \cdots & c_{t_j+1, t_{j+1}}e^{i\theta_{t_j+1, t_{j+1}}}I_k\\ \cdots & \cdots & \cdots \\ c_{t_{j+1},t_j+1}e^{i\theta_{t_{j+1},t_j+1}}I_k & \cdots &c_{t_{j+1},t_{j+1}}e^{i\theta_{t_{j+1},t_{j+1}}}I_k\end{array}\right ].$$

Thus, if
\begin{equation}\label{matrix U}
\mathscr{U}=\left [\begin{array}{cccccccc} U_1 &0 &\cdots &0&0&\cdots &\cdots&0 \\ 0&U_2&\cdots &0&0&\cdots &\cdots& 0\\ \cdot &\cdot & \cdot &\cdot &\cdot &\cdot &\cdot & \cdot \\0 & 0 &\cdots & U_s&0 &\cdots &\cdots &0\\ 0&0&\cdots &0 &I_k&0 &\cdots &0 \\ 0 & 0 & \cdots &0 & 0 &I_k& \cdots &0  \\\cdot&\cdot &\cdot &\cdot &\cdot &\cdot &\cdot &\cdot \\0 &0 & \cdots &0 &0&0&\cdots &I_k \end{array}\right ], 
\end{equation}
then $\mathscr{U}$ is unitary, commutes with $A_1$,  and $\mathscr{U}A_2\mathscr{U}^*$ consists of $k\times k$ blocks, each of which is a scalar multiple of $I_k$. Of course, this implies that for every $1\leq j\leq k$ the subspace 
\begin{equation}\label{L j}
\mathscr{L}_j=span \{e_j,e_{j+k},e_{j+2k},...,e_{j+(n-1)k} \}\end{equation} 
is invariant under the action of both $A_1$ and $A_2$, and the restrictions of both $A_1$ and $A_2$ to $\mathscr{L}_j$ are the same for all $j$,
$$A_1\Big|_{\mathscr{L}_j} = diag (\lambda_1,...,\lambda_n), \ A_2\Big|_{\mathscr{L}_j}=\left [ e^{i\theta_{rl}}c_{rl}\right ]_{r,l=1}^n.$$ 
The proof of Theorem \ref{main theorem} in the case $m=2$ is finished.

\textbf{Remark}. If the polynomial $R$ in \eqref{small spectrum} is irreducible, then  the partition claimed in  Lemma \ref{all i j}, in fact, consists of only 1 element (so it is not a partition). Indeed,
for $1\leq p \leq s$ \ let  
$$E_p=\left [ \begin{array}{cccc}\lambda_{t_p+1}I_k& 0 &\cdots &0 \\ 0 &\lambda_{t_p+2}I_k & \cdots &0\\ \cdot & \cdot & \cdot & \cdot \\ 0 & 0 & \cdots &\lambda_{t_{p+1}}I_k \end{array} \right ]. $$ 
Then  \eqref{blocks of A 2} yields 
\small\begin{eqnarray}
&\Big\{R^k(x_1,x_20=0\Big\}=\sigma_p(A_1,A_2)=\Big\{det (x_1A_1+x_2A_2-I_{nk})=0\Big\} \nonumber \\
&  \label{factor}   \\
&=\bigg\{ \bigg(\displaystyle \prod_{p=1}^s det(x_1E_p+x_2D_p-I_{r_p})\bigg)\bigg(\displaystyle \prod_{l=t_s+1}^n (\lambda_s x_1+d_sx_2-1)\Big)^k=0 \bigg \}.\nonumber
\end{eqnarray}\normalsize
If $t_s=r_1+\cdot +r_s<n$, then every $(x_1,x_2)$ that annihilates
$(\lambda_nx_1+d_nx_2-1) $ belongs to the zero set of $R$. Since $R$ is irreducible, that means that the lineal polynomial $(\lambda_nx_1+d_nx_2-1)$ is divisible by $R$, so $R$ must be of degree 1. But in this case by \cite{CSZ} matrices $A_1$ and $A_2$ commute, and, as mentioned above, in this case all $c_{ij}$ vanish for $i\neq j$, a contradiction. Thus, $r_1+\cdots +r_s=n$ and 
\begin{equation}\label{A 2 final}
A_2 =\left [\begin{array}{cccc} D_1 &0 &\cdots   & 0 \\ 0 &  D_2 &\cdots & 0\\ \cdot & \cdot &\cdot & \cdot  \\ 0 & 0  & \cdots & D_s \end{array} \right ].
 \end{equation}
Therefore, in this case
\begin{equation}\label{matrix U irreducible}
\mathscr{U}=\left [\begin{array}{cccc} U_1 &0 &\cdots &0 \\ 0&U_2&\cdots &0\\ \cdot &\cdot & \cdot &\cdot  \\0 & 0 &\cdots & U_s\end{array}\right ], 
\end{equation}
It easily follows from \eqref{A 2 final} and \eqref{matrix U irreducible} that, if $s>1$, then each subspace \eqref{L j} is a direct sum of subspaces of dimensions $r_1,...r_s$, and each of these subspaces is invariant with respect to $A_1$ and $A_2$. Since the restrictions of $A_1$ and $A_2$ are the same for all $j$, 
the characteristic polynomial of the restrictions of $A_1$ and $A_2$ to  $\mathscr{L}_j$ is independent of $j$ and is equal to $R$.  Since $\mathscr{L}_j$ is a direct sum of common invariant subspaces for $A_1$ and $A_2$, \ $R$ must be a product of $s$ factors, which contradicts its irreducibility. Thus, $s=1$, which means that, if $R$ is irreducible, then $\widehat{\mathscr{I}}= \widetilde{\mathscr{I}}_1$.

\subsubsection{$\bf{m>2}$} 

%\vspace{.2cm}

Again, suppose that $\lambda_1,...,\lambda_n$ are eigenvalues of $A_1$ each of multiplicity $k$. Since the whole tuple $(A_1,...,A_m)$ satisfies the conditions of Theorem \ref{main theorem}, each pair $(A_1,A_l)$ satisfies them as well. Therefore, as proved in the previous subsection,   each $A_l, \ l=2,...,m$ has a block structure comprised of $k\times k$ matrices $c_{ij}^lu_{ij}^l$ in an eigenbasis where $A_1$ is given by \eqref{A 1}, and the matrices $u_{ij}^l$ are unitary with $u^l_{ii}=I_k$. 

\vspace{.2cm}

The following Lemma is an analog of Lemma  \ref{relations}. 
\vspace{.2cm}

\begin{lemma}\label{relations1}
Let $j_0,...j_r$ be distinct numbers between $1$ and $n$, and $2\leq s_1,...,s_{r+1}\leq m$. Set $j_{r+1}=j_0.$  If $\displaystyle \prod_{l=0}^r c_{j_lj_{l+1}}^{s_{l+1}}\neq 0$, 	then
\begin{equation}\label{cycles1}
\displaystyle \prod_{l=0}^r u_{j_l j_{l+1}}^{s_{l+1}}=e^{i\theta}I_k .	
\end{equation}
	
\end{lemma}

The details of the proof of this  Lemma are similar to those of Lemma \ref{relations}, The only difference is that instead of
$$W_{j_1...j_r}=A_2\pr_{j_1}A_2\pr_{j_2}\cdots A_2\pr_{j_r}A_2 $$
which we used in  the proof of Lemma \ref{relations}, here we use
$$W_{j_1...j_r}^{s_1...s_{r+1}}=A_{s_1}\pr_{j_1}A_{s_2}\pr_{j_2}\cdots A_{s_r}\pr_{j_r}A_{s_r+1}. $$
For this reason we omit the proof.

We use Lemma \ref{relations1} to slightly modify the previous subsection's process  of constructing  a unitary matrix $\mathscr{U}$ in the form \eqref{matrix U},  so that  for all $1\leq l\leq m$,  \ $\mathscr{U}A_l\mathscr{U}^*$ consists of $k\times k$ blocks, each of which is a scalar multiple of $I_k$.

 Fix $(i,j), \ i\neq j$, and let $S_{ij}=\{l: \ c_{ij}^l\neq 0\}$. If  $S_{ij}=\emptyset $, we do not include $(i,j)$ and $(j,i)$ into $\mathscr{I}$. Otherwise, we include $(i,j)$ and $(j,i)$ into $\mathscr{I}$ and change matrices $u_{ij}^l$ and constants $c_{ij}^l$ in the following way. Choose $l_{ij}\in S_{ij}$. If $l\in S_{ij,}, \ l\neq l_{ij}$, we have by Lemma \ref{relations1} 
$$u_{ij}^{l_{ij}}u_{ji}^l=e^{i\theta _{ij}^L} I_k.$$
We set the new $u_{ij}^l$ to be $u_{ij}^{l_{ij}}$ and the new  $c_{ij}^l$ to be the old one times $e^{i\theta_{ij}^l}$. Thus, the product $c_{ij}^lu_{ij}^l$ remains the same. 

If $l\notin S_{ij}$, then $c_{ij}^l=0$, and we define $u_{ij}^l=u_{ij}^{l_{ij}}$. Once agin the product $c_{ij}^lu_{ij}^l$ remains the same (in this case 0). 

We use Lemma \ref{relations1} to verify the consistency of this definition. If $c_{ij}^l=0$ and for some $s$ we have $c_{si}^l$ and $c_{sj}^l$ do not vanish, then by \eqref{cycles1}

$$ u_{si}^lu_{ij}^{l_{ij}}u_{js}^l=e^{i\theta}I_k,$$
so that
$$u_{is}^lu_{sj}^l=e^{-i\theta}u_{ij}^{l_{ij}}. $$
Since $c_{ij}^l=0$, we see that 
$$c_{ij}^lu_{ij}^{l_{ij}}=c_{ij}^lu_{is}^lu_{sj}^l=0.$$

As a result of this step we formed the set $\mathscr{I}$ of pairs $(i,j)$ where $u_{ij}^l$ is defined. This set is the same for all $l=2,...,m$ and $u_{ij}^l$ depend only on $(i,j)$ and are independent of $l$. Also, for each $l$ all pairs $(i,j)$ for which $c_{ij}^l\neq 0$ are included in  the set  $\mathscr{I}$, and \eqref{cycles1} holds when all pairs $(j_l,j_{l+1})$ are in $\mathscr{I}$.  

Next, 
%for each $l$ 
we go through the steps of extending the  set $\mathscr{I}$ and defining $u_{ij}^l$ for added $(i,j)$, which were described in the previous subsection.  Since $\mathscr{I}$ is the same for all $l$ and $u_{ij}^l$ didn't depend on $l$ when $(i,j)\in \mathscr{I}$, the final maximal  set $\widehat{\mathscr{I}}$ and matrices $u_{ij}^l$ are independent of $l$ as well. Now, we define $\mathscr{U}$ by \eqref{matrix U}. Then for each $l$ the matrix $\mathscr{U}A_l\mathscr{U}^*$ consists of $k\times k$ blocks, each of which is a scalar multiple of $I_k$. Of course, the subspaces $\mathscr{L}_j$ given by \eqref{L j} are  invariant under the action of each $A_l$, and the restriction of each $A_l$ to all $\mathscr{L}_j$ are  the same. The proof of Theorem \ref{main theorem} is complete. 
\ \ \ \ \ \ \ \ \ \ \ \ \ \ \ \ \ \ \ \ \ \ \ \ \ \ \ \ \ \ \ \ \ \ \ \ \ \ \ \ \ \ \ $\Box$

\vspace{.3cm}

Note that each word $W\in \mathscr{W}$ that appeared in the statement of Theorem \ref{main theorem}, when considered as an element of the free algebra generated by $A_1,...,A_m$,  has degree of less than or equal to 
$n^2-n+1$. The following result is a Corollary to Theorem \ref{main theorem}

Let $(A_1,...,A_m)$ be a tuple of $nk\times nk$ \ self-adjoint matrices which may not be admissible, and let $\mathscr{V}$ be the collection of all non-commutative monomials of degree less than or equal to $n^2-n+1$ in the free algebra generated by $A_1,...,A_m$. We denote by $\sigma_p(\mathscr{V})$ the proper projective joint spectrum of the matrices in $\mathscr{V}$.
   Once again, as above, without loss of generality we may assume that all matrices $A_1,...,A_m$ are invertible.

\begin{corollary}\label{cor}
Let $A_1,...,A_m$ be tuple of self-adjoint invertible $nk\times nk$ matrices. Suppose that
$$\sigma_p(A_1,...,A_m)=\Big\{ (x_1,..,x_m)\in \C^m: \ \big(R(x_1,...,x_m)\big)^k=0\Big\}, $$
where $R$ is a polynomial of degree $n$ such that $R(x_1,0,...,0)$ has simple roots as a polynomial in $x_1$. The tuple $(A_1,...,A_m)$ is unitary equivalent to a tuple that is a direct sum of $k$ identical $m$-tuples of $n\times n$ matrices, if and only if
\begin{equation}\label{sigma V}
\sigma_p(\mathscr{V}) 
=\bigg\{\bigg({\mathscr R}_\mathscr{V}(x_1,...,x_\mathscr{M})\bigg)^k=0\bigg\}, 
\end{equation}	
where $\mathscr{M}$ is the total number of monomials in $\mathscr{V}$ and ${\mathscr R}_\mathscr{V}$ is a polynomial of degree $n$.

\end{corollary}

\begin{proof}
By Theorem \ref{admissible tr} there is a close to the identity admissible transformation $C$  
$$ \widehat{A}=CA, \ \widehat{A}=(\widehat{A}_1,...,\widehat{A}_m), \ A=(A_1,...,A_m)$$
that maps our tuple $A$ into an admissible tuple. Let us denote by $\widehat{\mathscr{W}}$ the set of words \eqref{W} obtained from the tuple $\widehat{A}$ instead of $A$. Since $C$ is close to the identity, the eigenvalues of $\widehat{A}_1$ have the same multiplicities. Each word in $\widehat{\mathscr{W}}$ is a linear combination of words in $\mathscr{V}$, and the degree of this word as an element of the algebra generated by $A_1,...,A_m$ does not exceed $n^2-n+1$. 
By \eqref{spectrum transform} and \eqref{sigma V}
$$\sigma_p(\widehat{\mathscr{W}}) 
=\bigg\{\bigg({\widehat{\mathscr R}}(x_1,...,x_\mathscr{M})\bigg)^k=0\bigg\}  $$
for some polynomial $\widehat{\mathscr{R}}$ of degree $n$, which, of course, implies that the tuple $(\widehat{A}_1,...,\widehat{A}_m)$ satisfies the conditions of Theorem \ref{main theorem}. 
Since the lattices of common invariant subspaces of the tuples $A$ and $\widehat{A}$ are the same, the result follows.
\end{proof}

\end{document}